\documentclass[12pt]{amsart}
\usepackage{jdr-style, jdr-macros}
\usepackage[all]{xy}

\title[Abel--Jacobi for compact-type models]%
  {A rigid analytic proof that the Abel--Jacobi map extends to compact-type models}
\author{Taylor Dupuy and Joseph Rabinoff}

\begin{document}

\begin{abstract}
  Let $K$ be a non-Archimedean valued field with valuation ring $R$.  Let $C_\eta$ be a $K$-curve with compact type reduction, so its Jacobian $J_\eta$ extends to an abelian $R$-scheme $J$.  We prove that an Abel--Jacobi map $\iota\colon C_\eta\to J_\eta$ extends to a morphism $C\to J$, where $C$ is a compact-type $R$-model of $J$, and we show this is a closed immersion when the special fiber of $C$ has no rational components.  To do so, we apply a rigid-analytic ``fiberwise'' criterion for a finite morphism to extend to integral models, and geometric results of Bosch and L\"utkebohmert on the analytic structure of $J_\eta$.
\end{abstract}

\maketitle

\section{Introduction}
Let $K$ be a field equipped with a nontrivial, non-Archimedean valuation $\val\colon K\to\R\cup\{\infty\}$.  Let $R$ be the (potentially non-noetherian) valuation ring of $K$, and let $k$ be its residue field.  We assume for simplicity that $k$ is algebraically closed.  Let $C\to\Spec(R)$ be a proper, flat relative curve, with generic fiber $C_\eta$ and special fiber $C_s$.  We suppose that $C$ is smooth and geometrically connected, and that $C_s$ is of \defi{compact type}, which means that it is reduced with (at worst) nodal singularities, and removing any node disconnects $C_s$.  Let $\{C_v\}_v$ denote the irreducible components of $C_s$ (the subscript $v$ is meant to denote a vertex in the incidence graph of $C_s$).

Let $J = \Pic^0_{C/R}$ be the identity component of the relative Picard scheme of $C$; this is smooth and separated over $R$ by~\cite[Theorem~9.4/1]{bosch_lutkeboh_raynaud90:neron_models}.  Its generic fiber is the Jacobian of $C_\eta$.  Let $J(C_v)$ be the Jacobian of $C_v$.  By~\cite[9.2, Example~8]{bosch_lutkeboh_raynaud90:neron_models}, we have an isomorphism on special fibers
\[ J_s = \Pic^0_{C_s/k} \overset\sim\To \prod_v J(C_v), \]
which is functorially induced by the inclusions $C_v\inject C_s$.  Since $J_s$ is an abelian variety, it follows from~\cite[Corollaire~5.5.2]{egaIII_1} that $J\to\Spec(R)$ is proper (the generic fiber $J_\eta$ is connected, so the $Z'$ in the statement of that result must be empty).  Hence $J$ is an abelian scheme.  It follows that when $R$ is noetherian, then $J$ is the N\'eron model of its generic fiber $J_\eta$.

Suppose that $C(K)\neq\emptyset$.  Let $\iota_\eta\colon C_\eta\to J_\eta$ be an Abel--Jacobi map defined by some choice of base point.  The goal of this note is to prove the following theorem.

\begin{thm}\label{thm:extend.abel.jacobi}
  The Abel--Jacobi map $\iota_\eta\colon C_\eta\to J_\eta$ extends, necessarily uniquely, to a morphism $\iota\colon C\to J$.  This morphism is a closed immersion if $C_s$ has no rational components.
\end{thm}

This result is known to experts using moduli-theoretic techniques, at least when $R$ is noetherian: see for instance,~\cite{caporaso07:abel_maps,caporaso_harris98:plane_curves,caporaso_esteves07:abel_maps_stable_curves}. (When $R$ is not noetherian, then there is no good notion of regularity for $C$, so existing results are not immediately applicable.)
% Using such techniques, one might proceed as follows.  By general formalism, the problem of extending the Abel--Jacobi map is equivalent to the problem of extending the line bundle $\sO(\Delta - \pi_2^{-1}(\text{base point}))$ on $C_\eta\times_K C_\eta$ to a family of multidegree-$0$ line bundles on $\pi_1 : C \times_R C \to C$.  Here $\Delta\subset C_\eta\times_K C_\eta$ is the diagonal and $\pi_1,\pi_2\colon C\times_R C$ are the projections.  The obvious extension, i.e.\ the same sheaf considered as a sheaf on $C \times_R C$, fails both to be a line bundle and to have fiberwise multidegree 0, so the heart of the proof would involve modifying the obvious extension so the two conditions are satisfied.
We will prove Theorem~\ref{thm:extend.abel.jacobi} by passing to a formal-analytic category and using Bosch--L\"utkebohmert's careful analysis of reductions of curves and Jacobians.  More specifically, we show in Proposition~\ref{prop:analytic.extension.criterion} that a morphism $f$ between the generic fibers of formal analytic varieties extends to a morphism of integral models if and only if $f$ takes formal fibers into formal fibers.  We then use results of Bosch--L\"utkebohmert to the effect that every formal fiber of a semistable curve is an open disc or an open annulus, and that every disc or annulus in the analytification of a Jacobian with good reduction is contained in a single formal fiber (a consequence of their so-called ``Homotopy Theorem''~\cite[Theorem~3.5]{bosch_lutkeboh84:stable_reduction_II}).  An advantage of working in this context is that our proof avoids noetherian hypotheses.

\begin{rem*}
  When $R$ is noetherian, then the N\'eron mapping property implies that $C_\eta\to J_\eta$ extends to the smooth locus of $C$.  Much of the content of Theorem~\ref{thm:extend.abel.jacobi} is that this morphism extends to the nodes.
\end{rem*}

\subsection{Acknowledgments}
The authors are very grateful to Jesse Kass for many helpful comments on an early draft, and to Lucia Caporaso for her expertise.

Dupuy was supported by the European Research Council under the European Union's Seventh Framework Programme (FP7/2007-2013) / ERC Grant agreement no.\ 291111 / MODAG. Rabinoff was supported by NSF DMS-1601842.  

\subsection{Reduction step}
Our analytic criteria are geometric in nature, so first we reduce to the case when the ground field is complete and algebraically closed.  

\begin{lem}\label{lem:extend.after.basechange}
  Let $K'$ be a valued field extension of $K$, and let $R'$ be the ring of integers of $K'$. Let $X,Y$ be flat $R$-schemes, and let $X' = X\tensor_R R'$ and $Y' = Y\tensor_R R'$.  Let $f_\eta\colon X_\eta\to Y_\eta$ be a morphism on generic fibers.  Then $f_\eta$ extends, necessarily uniquely, to a morphism $f\colon X\to Y$, if and only if the base change $f'_\eta\colon X'_\eta\to Y'_\eta$ extends to a morphism $f'\colon X'\to Y'$.  If $f'$ is a closed immersion, then so is $f$.
\end{lem}

\begin{proof}
  Since $X$ is $R$-flat, $X_\eta$ is dense in $X$, so if $f$ exists then it is unique.  Suppose that there exists $f'\colon X'\to Y'$ extending $f_\eta'$. The two compositions $X'\times_X X'\rightrightarrows Y'\to Y$ coincide when restricted to $(X'\times_X X')\tensor_{R}K$ because $f'_\eta = f_\eta\tensor_K K'$, so since $X'\times_X X'$ is $R'$-flat, the maps $X'\times_X X'\rightrightarrows Y$ coincide.  Thus $f$ exists by faithfully flat descent of morphisms.  The property of being a closed immersion can also be checked after faithfully flat base change.
\end{proof}

Replacing $K$ by the completion of an algebraic closure of the completion of $K$, we may and do assume from now on that $K$ is complete and algebraically closed.

\section{A fiberwise criterion for a finite morphism to extend to models}

Let $X$ be a proper, flat $R$-scheme.  The valuative criterion of properness provides us with a functorial reduction map $\pi\colon X(K)\to X(k)$.  A fiber of $\pi$ is called a \defi{formal fiber}.  By functoriality, if $f\colon X\to Y$ is a morphism of proper, flat $R$-schemes, then $f_\eta$ takes formal fibers of $X$ into formal fibers of $Y$.  In this section, we prove a partial converse statement: that if $f_\eta\colon X_\eta\to Y_\eta$ is a finite morphism that takes formal fibers of $X$ to formal fibers of $Y$, then $f_\eta$ extends to a morphism $f\colon X\to Y$, assuming the special fibers of $X$ and $Y$ are reduced.

\subsection{Admissible formal schemes}
In the situation discussed above, we wish to work more generally with admissible formal $R$-schemes.

Fix $\varpi\in K$ with $0 < \val(\varpi) < \infty$.  Let $X$ be a proper, flat $R$-scheme of finite presentation.  Then the $\varpi$-adic completion $\hat X$ of $X$ is a proper admissible formal $R$-scheme in the sense of~\cite{bosch_lutkeboh93:formal_rigid_geometry_I}. That is, $X$ is covered by formal affines $\Spf(A)$ for $A$ a flat, $\varpi$-adically complete $R$-algebra of topological finite presentation.  The admissible formal $R$-scheme $X$ has a ``generic fiber'' $X_\eta$, which is a rigid analytic variety over $K$.  The generic fiber is locally given by $\Spf(A)_\eta = \Sp(A\tensor_{R}K)$, which makes sense as $A_{K}:=A\tensor_{R}K$ is a quotient of a Tate algebra over $K$.  When $X\to\Spec(R)$ is proper, the analytic generic fiber $\hat X_\eta$ is canonically isomorphic to $X_\eta^\an$, the analytification of the generic fiber of $X$.  See~\cite[Theorem~A.3.1]{conrad99:irred_comps}.

\begin{lem}\label{lem:extend.after.completion}
  Let $X$ and $Y$ be proper, flat $R$-schemes of finite presentation, and let
  $f_\eta\colon X_\eta\to Y_\eta$ be a morphism on generic fibers.  Then
  $f_\eta$ extends to a morphism $f\colon X\to Y$ if and only if
  $f_\eta^\an\colon X_\eta^\an\to Y_\eta^\an$ extends to a morphism
  $\hat f\colon\hat X\to\hat Y$.  More precisely, $f_\eta$ extends if and only if there exists $\hat f\colon \hat X\to\hat Y$ such that $f_\eta^\an$ is the generic fiber of $\hat f$ under the identification $\hat X_\eta\cong X_\eta^\an$.
\end{lem}

\begin{proof}
  Consider the following commutative square:
  \[\xymatrix @R=.25in{
    {\Mor(X,Y)} \ar[r] \ar[d]_\cong & {\Mor(X_\eta,Y_\eta)} \ar[d]^\cong \\
    {\Mor(\hat X,\hat Y)}
    \ar[r] & {\Mor(X_\eta^\an, Y_\eta^\an)} }\]
  All arrows arise from functoriality of the operations of completion,
  analytification, and passing to the analytic generic fiber, along with the
  identifications $X_\eta^\an\cong\hat X_\eta$ and $Y_\eta^\an\cong\hat Y_\eta$.
  The left vertical arrow is a bijection by admissible formal
  GAGA~\cite[Proposition~10.3.1]{fujiwara_kato14:fgrI}, and the right vertical
  arrow is a bijection by rigid analytic GAGA.  This proves the lemma.
\end{proof}

\subsection{Formal analytic varieties}
For our purposes, we define a \defi{formal analytic variety} over $R$ to be a
separated admissible formal $R$-scheme $X$ with reduced special fiber $X_s$.
We regard the analytic generic fiber $X_\eta$ as a rigid space over $K$ as
above, and we regard the special fiber $X_s$ as a variety over $k$ and not
as a scheme---that is, we only consider the closed points of $X_s$.  With
these conventions, there exists a canonical, functorial, surjective
\defi{reduction map}, which we denote by $\pi\colon X_\eta\to X_s$.  A
\defi{formal fiber} of $X$ is the $\pi$-inverse image of a (closed) point of
$X_s$.  It is known that for every formal affine $U = \Spf(A)\subset X$, the
generic fiber $U_\eta = \Sp(A_{K})\subset X_\eta$ is an affinoid domain with
underlying set $\pi^{-1}(U_s)$, and that $A$ coincides with the ring
$A_{K}^\circ$ of power-bounded elements in the affinoid algebra $A_{K}$.  See
for example~\cite[\S1]{bosch_lutkeboh85:stable_reduction_I}.  In this situation
we say $\Spf(A)$ is the \defi{canonical model} of $\Sp(A_{K})$.  It follows
that $X$ is obtained by gluing the canonical models of a number of affinoid
domains in $X_\eta$.

We have the following useful criterion for a finite morphism of rigid spaces to
extend to a morphism of formal analytic varieties.

\begin{prop}\label{prop:analytic.extension.criterion}
  Let $X$ and $Y$ be formal analytic varieties over $R$, and let
  $f_\eta\colon X_\eta\to Y_\eta$ be a finite morphism.  Then $f_\eta$ extends to a
  morphism of formal analytic varieties $f\colon X\to Y$ if and only if there
  exists a set-theoretic map $f_s\colon X_s\to Y_s$ making the following square
  commute: 
  \[\xymatrix @R=.2in{ {X_\eta} \ar[r]^{f_\eta} \ar[d]_\pi &
    {Y_\eta} \ar[d]^\pi \\
    {X_s} \ar[r]_{f_s} & {Y_s}
  }\]
  In this case, $f_s$ coincides with the special fiber of $f$.
\end{prop}

In other words, $f_\eta$ extends to a morphism of formal analytic varieties if and only if it takes formal fibers of $X$ into formal fibers of $Y$.  It will be clear from the proof that the extension $f$ is unique if it exists.  Note that the morphism $f_s$ need not be finite; for instance, $f_\eta$ could be the identity, and $f$ an admissible formal blow-up on the special fiber.

\begin{proof}
  The stated condition is necessary by functoriality of the reduction map.
  Suppose now that such $f_s$ exists.  First we claim that $f_s$ is continuous.
  Let $V\subset Y$ be a formal affine, and let $U_s = f_s^{-1}(V_s)$ and
  $U_\eta = f_\eta^{-1}(V_\eta) = \pi^{-1}(U_s)$.  Note that $U_\eta$ is affinoid
  since $f_\eta$ is finite.  Let $U'\subset X$ be a formal affine.  Then
  $U'_\eta\cap U_\eta$ is an affinoid subdomain of $U'_\eta$ which is a union of
  formal fibers.  It follows
  from~\cite[Lemma~5.7]{bosch_lutkeboh84:stable_reduction_II} that
  $U'_\eta\cap U_\eta$ is the $\pi$-inverse image of a Zariski-open subset of
  $U'_s$, so that $U_s'\cap U_s$ is open.  As $U'$ was arbitrary,
  this establishes continuity of $f_s$.

  Let $V = \Spf(B)\subset Y$ be a formal affine and let $U = \Spf(A)$ be a
  formal affine with $U_s\subset f_s^{-1}(V_s)$.  Then
  $U_\eta\subset f_\eta^{-1}(V_\eta)$, so $f_\eta$ induces a homomorphism
  $\phi\colon B_{K}\to A_{K}$.  A homomorphism of affinoid algebras is automatically
  continuous, so $\phi$ takes $B = B_{K}^\circ$ into $A = A_{K}^\circ$, and thus
  induces a morphism $U\to V$.  This construction is compatible with
  restriction, so it glues to a map $f\colon X\to Y$.
\end{proof}

\begin{rem}
  A formal analytic variety $X$ may be regarded as a ringed space with underlying set
  $X_\eta$, whose open sets are the $\pi$-inverse images of Zariski-open subsets
  of $X_s$, and whose structure sheaf when evaluated on $\pi^{-1}(U_s)$ is the
  ring of power-bounded analytic functions on $U_\eta$ for a formal affine
  $U\subset X$.  See \cite[\S1]{bosch_lutkeboh85:stable_reduction_I}.
  The second paragraph of the proof of
  Proposition~\ref{prop:analytic.extension.criterion} is the simple
  fact that a morphism of rigid spaces $f_\eta\colon X_\eta\to Y_\eta$ extends to a
  morphism $f\colon X\to Y$ if and only if $f_\eta$ is continuous with respect to
  the corresponding formal analytic topologies.
\end{rem}

\section{Jacobians with good reduction}
We return to the notations of the introduction (still assuming $K$ complete and algebraically closed).  In particular, $C$ is a compact-type curve over $R$, and $J = \Pic^0_{C/R}$.  The special fiber $J_s$ is smooth, hence reduced, so we may view the completion $\hat J$ as a formal analytic variety.

The following result of Bosch--L\"utkebohmert is the main geometric input into
the proof of Theorem~\ref{thm:extend.abel.jacobi}.  

\begin{prop}\label{prop:annulus.into.formal.fiber}
  Let $U$ be an analytic open disc or open annulus.  Then the image of any
  morphism $f\colon U\to {J}_\eta^\an$ is contained in a single formal fiber.
\end{prop}

\begin{proof}
  This follows from~\cite[Propositions~5.11 and~5.13]{bosch_lutkeboh84:stable_reduction_II}, as an open disc (resp.\ annulus) is an increasing union of affinoid discs (resp.\ annuli).
\end{proof}

\begin{proof}[Proof of Theorem~\ref{thm:extend.abel.jacobi}, step 1]
  At this point we can  show that $C_\eta\to J_\eta$ extends to a morphism $C\to J$; we will prove below that this morphism is a closed immersion when $C_s$ has no rational components.  By Lemma~\ref{lem:extend.after.completion}, we may replace $C$ and $J$ by their completions to work in the category of formal analytic varieties.  The theorem then follows from Propositions~\ref{prop:analytic.extension.criterion} and~\ref{prop:annulus.into.formal.fiber} because all formal fibers of $C$ are open discs and open annuli, by~\cite[Propositions~2.2 and~2.3]{bosch_lutkeboh85:stable_reduction_I}.
\end{proof}

\begin{rem*}
  The above argument actually shows that a finite morphism from $C_\eta$ to any Jacobian $J$ with good reduction extends to integral models.  One can replace $J$ with any abelian variety $A$ with good reduction if Proposition~3.1 holds for $A$; it is not clear if this is true, however.
\end{rem*}

Now that we know it exists, we can describe the behavior of the extension $C\to J$ on special fibers.  As passing to the completion of an algebraic closure of the completion of a field does not affect special fibers, the conclusions of Proposition~\ref{prop:AJ.special.fiber} are valid when $K$ is not complete and algebraically closed.

\begin{prop}\label{prop:AJ.special.fiber}
  Let $\iota\colon C\to J$ be the extension of an Abel--Jacobi map on the generic fiber, as provided by Theorem~\ref{thm:extend.abel.jacobi}.  Let $\iota_v\colon C_v\to J_s$ be the restriction of $\iota_s$ to the component $C_v$.  Under the identification $J_s = \prod_v J(C_v)$, for any component $C_w$, the morphism
  \[ \iota_w\colon C_w\to\prod_v J(C_v) \quad\text{is }
    \begin{cases}
      \text{zero} & \text{in the $v$-coordinate for $v\neq w$, and} \\
      \text{an Abel--Jacobi map} & C_w\inject J(C_w) \text{ in the $w$-coordinate,}
    \end{cases}\]
  up to translation on $J_s$.
\end{prop}

\begin{proof}
  This is a statement about smooth varieties over an algebraically closed field, so it may be checked on closed points.  After translation, we may assume that $\iota_\eta$ is the Abel--Jacobi map defined by a base point $P_0\in C(K)$ whose reduction $\td P_0\in C(k)$ is a smooth point of $C_w$.  For any other point $P\in C(K)$ whose reduction $\td P$ is in $C_w^\sm$, the reduction $\iota_w(\td P)$ of $\iota_\eta(P) = [P-P_0]$ has $w$-coordinate $[\td P-\td P_0]\in J(C_w)$ and has $v$-coordinate zero for $v\neq w$; see~\cite[Theorem~5.1(c)]{bosch_lutkeboh84:stable_reduction_II}.  Since every point of $C_w^\sm$ is the reduction of some point $P\in C(K)$, this shows that $C_w^\sm\to J(C_w)$ is the restriction of an Abel--Jacobi map, and that $C_w^\sm\to J(C_v)$ is zero for $v\neq w$.  The assertion follows because $C_w^\sm$ is dense in $C_w$.
\end{proof}

\begin{proof}[Proof of Theorem~\ref{thm:extend.abel.jacobi}, step 2]
  It remains to prove that $\iota\colon C\to J$ is a closed immersion when $C_s$ has no rational components.  First we show that $\iota$ is injective on closed points.  Fix an irreducible component $C_w$ of $C_s$, and translate $\iota$ so that $\iota_w\colon C_w\to\prod_v J(C_v)$ is as described in Proposition~\ref{prop:AJ.special.fiber}.  In particular, the $v$-coordinate of $\iota_w(\td P)$ is zero for $\td P\in C_w$ and $v\neq w$.  We claim that this condition characterizes $C_w$: i.e., that if $\td P\in C(k)$ and $\iota_s(\td P)$ has zero $v$-coordinate for all $v\neq w$, then $\td P\in C_w$.  This is enough to prove injectivity.

  Let $\td P\in C_{w'}(k)$ for $w\neq w'$.  Since the incidence graph of $C_s$ is a tree, there is a unique chain of components $C_w = C_{w_0},\,C_{w_1},\ldots,C_{w_r}=C_{w'}$ connecting $C_w$ and $C_{w'}$.  Suppose first that $r=1$, so that $C_w$ intersects $C_{w'}$.  Since $\td P\notin C_w$ and since $\iota_{w_1}$ is injective, we have that the $w_1$-coordinate of $\iota_s(\td P)$ is nonzero, as claimed.  If $r > 1$, then let $\td P'$ be the point where $C_{w_1}$ and $C_{w_2}$ intersect.  Then $\td P'\notin C_w$, so the $w_1$-coordinate of $\iota_s(\td P')$ is nonzero, as above.  But $\iota_{w_i}$ has constant $w_1$-coordinate for $i > 1$, so the $w_1$-coordinates of $\iota_s(\td P)$ and $\iota_s(\td P')$ coincide.  This completes the proof of injectivity.

  Since $\iota$ is quasi-finite and proper (as $C$ is proper), it is finite.  We now pass to formal completions to prove it is a closed immersion.  Both $\hat C$ and $\hat J$ are formal analytic varieties, and $\hat\iota\colon\hat C\to\hat J$ is an affine morphism.  By~\cite[Proposition~6.4.2/1]{bosch_guntzer_remmert84:non_archimed_analysis}, it is enough to prove that $\iota_s$ is a closed immersion.  This is a statement about a morphism of varieties over an algebraically closed field; we already know that $\iota_s$ is injective on points, so it suffices to show it is injective on tangent vectors.  This is clear for smooth points of $C_s$ by Proposition~\ref{prop:AJ.special.fiber}.  Let $\td P\in C(k)$ be a singular point, and let $C_v,C_w$ be the irreducible components containing $\td P$.  The tangent space at $\td P$ is two-dimensional, generated by tangent vectors in the $C_v$- and $C_w$-directions.  But these map to tangent vectors with zero $J(C_w)$- and $J(C_v)$-components, respectively, so their images under $\iota_s$ are linearly independent.
\end{proof}

\bibliographystyle{egabibstyle}
\bibliography{papers}

\end{document}